\documentclass[review]{elsarticle}

\usepackage{lineno,hyperref}
\usepackage[left=2cm,top=2.5cm,right=2cm,bottom=2.5cm]{geometry}
\usepackage{graphicx,array,delarray}
\usepackage{subfigure}
\usepackage{latexsym}
\usepackage{amscd}
\usepackage{amsfonts, amsmath, amssymb}
\usepackage[utf8]{inputenc}
\usepackage{verbatim}
\usepackage{mathrsfs}
\usepackage{fancyhdr}
\usepackage{float}
\usepackage{palatino}
\usepackage[english]{babel}
\usepackage[usenames]{color}

\modulolinenumbers[5]
\newcommand{\RR}{\mathbb{R}}

\newcommand{\NN}{\mathbb{N}}

\newcommand{\R}{\mathbf{R}}

\newcommand{\Real}{\R \mathbf{e}}

\newtheorem{theorem}{\textbf{Theorem}}

\newtheorem{definition}[theorem]{\textbf{Definition}}

\newtheorem{lemma}[theorem]{\textbf{Lemma}}

\newtheorem{proposition}[theorem]{\textbf{Proposition}}
\newtheorem{remark}[theorem]{Remark}

\newenvironment{proof}[1][Proof]{\noindent\textbf{#1.} }{\ \rule{0.5em}{0.5em}}

\begin{document}

\begin{frontmatter}

\title{\textbf{On the strong convergence of the Faedo-Galerkin approximations to a strong T-periodic solution of the torso-coupled bi-domain model}}

\author[mymainaddress0]{Raul Felipe-Sosa\corref{mycorrespondingauthor}}
\cortext[mycorrespondingauthor]{Corresponding author}
\ead{raul.felipe@unach.mx}
\author[mymainaddress]{Andres Fraguela-Collar}
\author[mymainaddress0]{Yofre H. García-Gómez}
%\author[mymainaddress]{Moises Soto-Bajo}
%%%%%%%%%%%%%%%%%%%%%%%%%%%%%%%%%%

\address[mymainaddress0]{Facultad de Ciencias en Física y Matemáticas-UNACH, Chiapas, M\'exico.}
\address[mymainaddress]{Facultad de Ciencias Físico Matemáticas-BUAP, Puebla, M\'exico.}

\begin{abstract}
In this paper, we investigate the convergence of the Faedo-Galerkin approximations, in a strong sense, to a strong T-periodic solution of the torso-coupled bidomain model where $T$ is the period of activation of the inner wall of heart. First, we define the torso-coupled bi-domain operator and prove some of its more important properties for our work. After, we define the abstract evolution system of equations associated with torso-coupled bidomain model and give the definition of strong solution. We prove that the Faedo-Galerkin's approximations have the regularity of a strong solution, and we find that some restrictions can be imposed over the initial conditions, so that this sequence of Faedo-Galerkin fully converge to a strong solution of the Cauchy problem. Finally, this results are used for showing the existence a strong $T$-periodic solution. 
\end{abstract}


\end{frontmatter}

\section{Introduction}
In this paper, we use the formulation of the bidomain equations given in \cite{Sundnes}, where a parabolic PDE coupled to an elliptic one is considered, and both coupled to a system of ODEs that model the activation variables for the ionic currents. In this case, the state variables are the membrane potential $u_m = u_i - u_e$ and the extra and intra cellular potential $u_i, u_e$, respectively. We consider a region $\Omega$ in $\RR^3$ to represent the heart embedded in the torso. In this regard, it is assumed that $\Omega = \Omega_H \cup \Omega_T$, where $\Omega_H$ is the region representing the heart and $\Omega_T$ represents the torso. The boundary of $\Omega_H$ has two components, that is, $\partial \Omega_H = \Sigma_{endo} \cup \Sigma_{epi}$ whit $\Sigma_{endo}$ representing the cardiac endocardium, where the periodic activation is given, and $\Sigma_{epi}$ represents the epicardium which is in contact with the torso. Also, $\Omega_T$ has two components, $\partial \Omega_T = \Sigma_{epi}\cup \Sigma_{T}$, where $\Sigma_{T}$ the exterior boundary of the torso. 

The equations of bidomain are performed as
\begin{align}
\label{eqn0}&A_mC_m\frac{\partial u_m}{\partial t} + A_m I_{ion}(u_m,\mathbf{w}) - \nabla \cdot\left(\sigma_i \nabla u_m\right) - \nabla \cdot\left(\sigma_e \nabla u_e\right) = 0, \;\;\; \hbox{in $(0,\infty)\times \Omega_H$},\\
\nonumber
&-\nabla\cdot\left(\sigma_i\nabla u_m + (\sigma_i + \sigma_e)\nabla u_e\right) = 0,\;\;\; \hbox{in $(0,\infty)\times \Omega_H$},\\
\nonumber
&\frac{\partial \mathbf{w}}{\partial t} + g(u_m,\mathbf{w}) = 0,\;\;\; \hbox{in $(0,\infty)\times \Omega_H$},
\end{align}
where,
\begin{itemize}
\item $A_m, C_m$ are the rate of cellular membrane area per volume unit and the cellular membrane capacitance per area unit, respectively,
\item $I_{ion}$ is the ionic current across the cell membrane per area unit which depends on the membrane potential and the vector of activation variables $\mathbf{w}$,
\item $g$ is a vector field that models the behavior of the ionic channels,
\item $\sigma_i(\mathbf{x}), \sigma_e(\mathbf{x})$ are matrices depending on the space variable$\mathbf{x} \in \Omega_H$.
\end{itemize}
In general, $\mathbf{w} = \left(w_1, w_2,\ldots,w_k\right)$ is a vector that have several components, called gate variables. These variables model the closure or opening of the ionic channels.

These equations are coupled with a model for the torso which are considered as a passive conductor. 
\begin{align*}
&-\nabla\cdot\left(\sigma_T \nabla u_T\right) = 0,\;\;\; \hbox{in $(0,\infty)\times \Omega_T$},\\
&\sigma_T\nabla u_T\cdot \mathbf{n} = 0,\;\;\; \hbox{in $(0,\infty)\times \Sigma_T$}.
\end{align*}
Here, $u_T$ is the potential in the torso, $\sigma_T$ is its conductivity which depends on the space variable, and $\mathbf{n}$ is the outward unit normal to $\Sigma_T$. The last condition means that no current can flow outside the torso across $\Sigma_T$. 

In addition to the equations of bidomain and the model for the torso, we should give boundary conditions and coupling conditions with the torso. A condition widely assumed, see \cite{Tung} and \cite{Boulakia}, is that the intracellular current does not propagate
outside the heart:
\begin{align*}
\sigma_i \nabla u_i \cdot \mathbf{n} = 0,\;\;\; \hbox{in $(0,\infty)\times \partial \Omega_H$},
\end{align*} 
which can be rewritten as 
\begin{align*}
\sigma_i \nabla u_m\cdot \mathbf{n} +  \sigma_i \nabla u_e\cdot \mathbf{n} = 0,\;\;\; \hbox{in $(0,\infty)\times \partial \Omega_H$}.
\end{align*}
For the extracelullar current, we give the following boundary condition
\begin{align*}
\sigma_e \nabla u_e\cdot \mathbf{n} = s_e,,\;\;\; \hbox{in $(0,\infty)\times\Sigma_{endo}$}. 
\end{align*}
where $s_e$ is a $T$-periodic function in the variable $t$ that represents the activation of endocardium. 

On other hand, perfect electric conditions between the heart and the torso are considered
\begin{align*}
\left\{
\begin{array}{c}
u_e = u_T,\\
\sigma_e \nabla u_e\cdot \mathbf{n} = \sigma_T \nabla u_T\cdot \mathbf{n},
\end{array}
\right.
\hbox{in $(0,\infty)\times\Sigma_{epi}$}.
\end{align*}
In summary, the Cauchy problem associated with the torso-coupled bidomain model can be presented as
\begin{align}
\label{eqn1}&\frac{\partial u_m}{\partial t} + f(u_m,\mathbf{w}) - \nabla \cdot\left(\sigma_i \nabla u_m\right) - \nabla \cdot\left(\sigma_e \nabla u_e\right) = 0, \;\;\; \hbox{in $(0,\infty)\times \Omega_H$},\\
\label{eqn2}&-\nabla\cdot\left(\sigma_i\nabla u_m + (\sigma_i + \sigma_e)\nabla u_e\right) = 0,\;\;\; \hbox{in $(0,\infty)\times \Omega_H$},\\
&\frac{\partial \mathbf{w}}{\partial t} + g(u_m,\mathbf{w}) = 0,\;\;\; \hbox{in $(0,\infty)\times \Omega_H$},\\
&\sigma_i \nabla u_m\cdot \mathbf{n} +  \sigma_i \nabla u_e\cdot \mathbf{n} = 0,\;\;\; \hbox{in $(0,\infty)\times \partial \Omega_H$},\\
&\sigma_i \nabla u_m\cdot\mathbf{n} + \left(\sigma_i + \sigma_e\right)\nabla u_e\cdot\mathbf{n} = s_e,\;\;\; \hbox{in $(0,\infty)\times\Sigma_{endo}$},\\
&\left\{
\begin{array}{c}
u_e = u_T,\\
\sigma_e \nabla u_e\cdot \mathbf{n} = \sigma_T \nabla u_T\cdot \mathbf{n},
\end{array}
\right.
\hbox{in $(0,\infty)\times\Sigma_{epi}$},\\
\label{eqn7}&-\nabla\cdot\left(\sigma_T \nabla u_T\right) = 0,\;\;\; \hbox{in $(0,\infty)\times \Omega_T$},\\
\label{eqn8}&\sigma_T\nabla u_T\cdot \mathbf{n} = 0,\;\;\; \hbox{in $(0,\infty)\times \Sigma_T$},\\
\label{eqn9}&u_m(0,\mathbf{x}) = u^{(0)}_m(\mathbf{x}),\;\; \mathbf{w}(0,\mathbf{x}) = \mathbf{w}^{(0)}(\mathbf{x}),\;\;\; \hbox{en $\Omega_H$}.
\end{align}
\begin{remark}
Note that the constants $A_m, C_m$ do not appear in the equation \eqref{eqn1}. This is owing to we have divided the equation \eqref{eqn0} by the term $A_m C_m$ and the terms $I_{ion}/C_m, \sigma_{i,e}/A_mC_m$ have been denoted by $f, \sigma_{i,e}$, respectively.
\end{remark}
In this paper, we work with the operational formulation of the bidomain problem. That is, the boundary problem \eqref{eqn1}-\eqref{eqn9} is described as an abstract evolution system, where its linear part is defined from the so-called torso coupled bidomain operator. Our goal is to obtain strong solutions associated to this formulation, taking them as limit of a function sequence which have the regularity of a strong solution. These are solutions are of a non-linear ODEs systems sequence. The mentioned sequence of functions is called Faedo-Galerkin approximation.

The results above statement, which are proved in this paper, present certain advantages with respect to other results found in the literature about the convergence the Faedo-Galerkin approximations. In general, the Faedo-Galerkin systems arise in the weak or variational setting. In that case, the principal result with respect to the convergence of the Faedo-Galerkin approximations affirms the convergence, in a weak sense, of a subsequence to a weak solution. However, in this paper, we demonstrate the convergence of the entire sequence to a strong solution. This result has important implications. For instance, starting from a result like this one, it is possible to define a method in order to approximate a strong solution. In this sense, the convergence of the entire sequence to the sought strong solution allows us to make sure that the functions of the sequence are every time closer to this solution, when the subindex towards to infinite. This would not be a fact if we only prove that one subsequence of the Faedo-Galerkin sequence converges.   
\section{Some general considerations}
In the strong context, we assume that boundary of $\Omega$, $\partial \Omega$, is $C^{2 + \nu}$ and the coefficients of matrices $\sigma_{i,e}, \sigma_{T}$ are $C^{1 + \nu}(\Omega_{H,T})$, for some $\nu > 0$, respectively. Also, the conductivity matrices are uniformly elliptic, that is, there are positive constants $M, m$ such that
\begin{align*}
m|\xi| \leq \xi^{t}\sigma_{i,e,T}(\mathbf{x})\xi \leq M|\xi|,\;\;\; \hbox{for all $\xi \in \RR^3$ and a.e $\mathbf{x} \in \Omega_{H,T}$}.
\end{align*} 
The fiber structure of the heart is another aspect that should be considered. This structure determines the anisotropic properties of the cardiac muscle. It is know that velocity of conduction of the depolarization wavefront is faster in the longitudinal direction to the fiber than in the transverse direction. This implies the existence of a matrix $M(\mathbf{x})$ whose columns are the orthogonal vectors $\left\{\mathbf{a}_l(\mathbf{x}),\mathbf{a}_t(\mathbf{x}),\mathbf{a}_{n}(\mathbf{x})\right\}$ with $\mathbf{a}_l$ parallel to the local fiber direction, $\mathbf{a}_{n}$ the vector orthogonal to the fiber lamina and $\mathbf{a}_t$ the vector lying on the lamina and orthogonal to fiber direction, such that 
\begin{align*}
\sigma_{i,e}(\mathbf{x}) = M^{T}(\mathbf{x})M^{*}_{i,e}M(\mathbf{x}).
\end{align*}
Setting 
\begin{align*}
M^{*}_{i,e} = \hbox{diag}\left[\sigma^{i,e}_{l},\sigma^{i,e}_{t},\sigma^{i,e}_{n}\right],
\end{align*}
where $\sigma^{i,e}_{l},\sigma^{i,e}_{t},\sigma^{i,e}_{n}$ are local coefficients of conductivity measured along the corresponding directions, see \cite{Colli1}. 

The fibers are tangent to the boundary, thus 
\begin{align}\label{ortfib1}
\sigma_{i,e}(\mathbf{x})\mathbf{n}(\mathbf{x}) = \sigma^{i,e}_{n}\mathbf{n}(\mathbf{x}),\;\;\; \hbox{for a.e. $\mathbf{x} \in \partial \Omega_H$},
\end{align}
with $\sigma^{i,e}_{n} \geq m > 0$.

The above has the following consequence
\begin{lemma}
If the boundary of $\Omega_H$ is $C^1$ and $M, M^{*}_{i,e}$ are $C^{0}(\overline{\Omega_H})$, then the following statements are true
\begin{description}
\item[i)] $\nabla u_i\cdot\mathbf{n} = 0$ and $\nabla u_e\cdot\mathbf{n} = 0$ $\Leftrightarrow$ $\nabla u_m\cdot\mathbf{n} = 0.$ 
\item[ii)] $\nabla u_{i,e}\cdot \mathbf{n} = 0$ $\Leftrightarrow$ $\left(\sigma_{i,e}\nabla u_{i,e}\right)\cdot \mathbf{n} = 0$ and $\left(\sigma_i\nabla u_{e}\right)\cdot \mathbf{n} = 0$ $\Leftrightarrow$ $\left\{\begin{array}{c}
\left(\sigma_i\nabla u_m\right) \cdot \mathbf{n} + \left(\sigma_i\nabla u_e\right) \cdot \mathbf{n} = 0,\\
\left(\sigma_i\nabla u_m\right)\cdot \mathbf{n} + \left((\sigma_i + \sigma_e)\nabla u_e\right)\cdot\mathbf{n} = 0.
\end{array}\right.$
\end{description}
\end{lemma} 
\begin{proof}
The first affirmation \textbf{i)} is immediate. Let us prove \textbf{ii)}. Note that, due to the symmetry of $\sigma_{i,e}$ and \eqref{ortfib1} we have 
\begin{align*}
\left(\sigma_{i,e}\nabla u_{i,e}\right)\cdot\mathbf{n} = \left(\sigma^{T}_{i,e}\mathbf{n}\right)\cdot\nabla u_{i,e} = \sigma^{i,e}_n \nabla u_{i,e}\cdot \mathbf{n}.
\end{align*}
The first equivalence in \textbf{ii)} it is immediately obtained from the above equalities. The second equivalence is easy obtained from the first equivalence.
\end{proof}
\section{Operational formulation. Strong periodic solutions}
In this section we obtain what we will call torso coupled bidomain operator. Once it is obtained, the boundary problem associated to the bidomain model is formulated. The key idea is that we can obtain $u_e$ from the membrane potential $u_m$ integrating the equation \eqref{eqn2} with suitable boundary conditions. This allows us to define a parabolic abstract evolution problem for the pair $(u_m,\mathbf{w})$ in the form
\begin{align*}
&\frac{du_m}{dt} + f(u_m,\mathbf{w}) + Au_m = s(t),\\
&\frac{d\mathbf{w}}{dt} + g(u_m,\mathbf{w}) = 0
\end{align*}
where $A$ is a integro-differential second order elliptic operator and $s(t)$ is a function defined in a real interval with values in certain functional space, that depend on the activation function $s_e$. From the solution of this problem we can calculate the total extracellular potential that we denote by
\begin{align*}
u = \left\{
\begin{array}{cc}
u_e, & \hbox{in $\Omega_H$},\\
u_T, & \hbox{in $\Omega_T$}.
\end{array}
\right.
\end{align*} 
We need to give the following notation. Given $X$, certain space of integrable function on a region $\Theta \subset \RR^{3}$, we define 
\begin{align*}
X/\RR = \left\{\phi \in X: \int_{\Theta} \phi = 0\right\}.
\end{align*}
We denote by 
\begin{align*}
&L_H = L_2(\Omega_H),\;\; V_{H} = H_1(\Omega_H), \;\; U_{H} := V_H/\RR,\\
&L_T = L_2(\Omega_T),\;\; V_{T} = H_1(\Omega_T), \;\; U_{T} := V_T/\RR,
\end{align*}
and
\begin{align*}
H = L_2(\Omega),\;\; V = \left\{\phi \in H_1(\Omega): \int_{\Omega_H}\phi = 0\right\}, \;\; U := V/\RR,
\end{align*}
where $\Omega = \Omega_H \cup \Omega_T$.

In the spaces $U_{H,T}$ and $U$ we use the equivalent norm to the norm in $V_{H,T}$ and $V$ given by $$|\phi|_{U_{H,T}} = \left(\int_{\Omega_{H,T}}\left|\nabla \phi\right|^2\right)^{1/2},$$ and $$|\phi|_{U} = \left(\int_{\Omega}\left|\nabla \phi\right|^2\right)^{1/2},$$ respectively. 

Here, each triple $\left\{U_{H,T}, L_{H,T}/\RR,U^{'}_{H,T}\right\}$, $\left\{V_{H,T}, L_{H,T},V^{'}_{H,T}\right\}$, $\left\{U, H/\RR,U^{'}\right\}$ and $\left\{V, H ,V^{'}\right\}$ is a Gelfand triple, see \cite{Rubi}. This means that we have
\begin{align*}
&U_{H,T} \subset L_{H,T}/\RR \subset U^{'}_{H,T}, \\
&V_{H,T} \subset L_{H,T} \subset V^{'}_{H,T},\\
&U \subset H/\RR \subset U^{'},\\
&V \subset H \subset V^{'},\\
\end{align*}
where the inclusions are continuous and the first inclusion is compact, in each case. 

For each duality pair $\left\{X,X'\right\}$, in the above Gelfand triplets, by 
$\left\langle u,\phi\right\rangle$, we denote  the value of functional $u \in X'$ in the element $\phi \in X$. The spaces $X$ and $X'$ will be recognized from the context. Also, by $(\cdot,\cdot)$, we denote  the inner product in $L_H$ or $H$. 

Furthermore, in that follows we use two operators: a restriction operator and other prolongation. By $R_H$, we denote the operator defined from $V$ into $U_H$ such that to each $\phi \in V \mapsto \left.\phi\right|_{\Omega_H} \in U_{H}$. $R_H$ is bounded, in fact $\left\|R_{H}\right\| \leq 1$ and its adjoin operator $R^{*}: U^{'}_{H} \rightarrow V^{'}$ is given by the relation
\begin{align*}
\left\langle R_H^{*}\phi_{H},\phi\right\rangle = \left\langle \phi_{H},R_{H}\phi\right\rangle,\;\;\; \hbox{for all $(\phi_{H}, \phi) \in U'_H\times V$}.
\end{align*}
If we assume that $\Omega_H$ has boundary $C^{1}$ from Theorem IX.7 of \cite{Rubi}, it is deduced that there is a prolongation operator, which is bounded,  $P_H: U_H \rightarrow V$, such that, $R_H P_H = I_d: U_H \rightarrow U_H$. Its adjoin operator $P_H^{*}: V' \rightarrow U'_H$ defined by
\begin{align*}
\left\langle P_H^{*}\phi,\phi_H\right\rangle_{U'_H\times U_H} = \left\langle \phi,P_H\phi_H\right\rangle_{V'\times V},\;\;\; \hbox{for all $(\phi, \phi_H) \in V\times U_H$}.
\end{align*}
In a similar way, we define the operator $R_T: U \rightarrow U_T$.

In that follows, we obtain the so-called torso-coupled bidomain operator that is the linear part of the abstract evolution problem that we want to define. In this sense, we study the following boundary problem
\begin{align}
\label{eqnn1}&-\nabla\cdot\left(\sigma_i \nabla u_m\right) - \nabla\cdot\left(\sigma_i \nabla u_e\right) = 0,\;\;\; \hbox{in $\Omega_H$},\\
&-\nabla\cdot\left(\sigma_i \nabla u_m + \left(\sigma_i + \sigma_e\right)\nabla u_e\right) = 0,\;\;\; \hbox{in $\Omega_H$},\\
&\sigma_i \nabla u_m\cdot\mathbf{n} + \sigma_i \nabla u_e\cdot\mathbf{n} = 0,\;\;\; \hbox{in $\partial\Omega_H$},\\
&\sigma_i \nabla u_m\cdot\mathbf{n} + \left(\sigma_i + \sigma_e\right)\nabla u_e\cdot\mathbf{n} = s_e,\;\;\; \hbox{in $\Sigma_{endo}$},\\
&\left\{
\begin{array}{c}
u_e = u_T,\\
\sigma_e\nabla u_e \cdot \mathbf{n} = \sigma_T\nabla u_T \cdot \mathbf{n},
\end{array}
\right.\;\;\; \hbox{in $\Sigma_{epi}$},\\
&-\nabla\cdot\left(\sigma_T \nabla u_T\right) = 0,\;\;\; \hbox{in $(0,\infty)\times \Omega_T$},\\
\label{eqnn2}&\sigma_T\nabla u_T\cdot \mathbf{n} = 0,\;\;\; \hbox{in $(0,\infty)\times \Sigma_T$}.
\end{align}
The variational formulation of the problem \eqref{eqnn1}-\eqref{eqnn2} has the following form
\begin{align}
\label{var0}&a_i(u_m,\phi_H) + a_i(R_Hu,\phi_H) = 0,\;\;\; \hbox{for all $\phi_H \in U_H$},\\
\label{var1}&\left(\tilde{a}_e + \tilde{a}_i\right)(u,\phi) + a_i(u_m,R_H\phi) = \left\langle \bar{s}_e,R_H \phi\right\rangle,\;\;\; \hbox{for all $\phi \in U$}.
\end{align}
where
\begin{align*}
u = \left\{
\begin{array}{cc}
u_e, & \hbox{in $\Omega_H$,}\\
u_T, & \hbox{in $\Omega_T$,}\\
\end{array}
\right.
\end{align*}
$$a_{i,e}(u_H,\phi_H) = \int_{\Omega_H} \sigma_{i,e}\nabla u_H \nabla \phi_H,$$ are bilinear forms defined in $U_H\times U_H$. Also, 
$$\tilde{a}_{i,e}(u,\phi) = \int_{\Omega} \tilde{\sigma}_{i,e}\nabla u \nabla \phi,$$ are bilinear forms defined in $U\times U$, where 
\begin{align*}
\tilde{\sigma}_{e} = \left\{
\begin{array}{cc}
\sigma_e & \hbox{in $\Omega_H$},\\
\sigma_T & \hbox{in $\Omega_T$},
\end{array}
\right. 
\end{align*}
and
\begin{align*}
\tilde{\sigma}_{i} = \left\{
\begin{array}{cc}
\sigma_i & \hbox{in $\Omega_H$},\\
0 & \hbox{in $\Omega_T$}.
\end{array}
\right. 
\end{align*}
Finally, $\bar{s}_e \in U'_H$ is defined by 
\begin{align*}
\left\langle \bar{s}_e,\phi_H\right\rangle := \int_{\Sigma_{endo}}s_e\phi_H,\;\;\; \forall \phi_H \in  U_H.
\end{align*}
The bilinear forms $\tilde{a}_{i,e}$ and $a_{i,e}$ are symmetric, continuous and coerciva. Because of this, by the Lax-Milgram theorem, for each $u_m \in U_H$ fixed there is an only $\tilde{u}$ that satisfies de variational equation \eqref{var1}. Besides, we have that there are operators 
\begin{align*}
&\tilde{A}_{i,e}: U \rightarrow U',\\
&A_{i,e}: U_H \rightarrow U'_H,
\end{align*}
which are one-to-one continuous with continuous inverse, so that 
\begin{align*}
&\left\langle \tilde{A}_{i,e}\phi_1,\phi_2\right\rangle_{U'\times U} = \tilde{a}_{i,e}(\phi_1,\phi_2),\;\;\; \hbox{for all $\phi_1, \phi_2 \in U$},\\
&\left\langle A_{i,e}\phi^{(1)}_H,\phi^{(2)}_H\right\rangle_{U'\times U} = \tilde{a}_{i,e}(\phi^{(1)}_H,\phi^{(2)}_H),\;\;\; \hbox{for all $\phi^{(1)}_H, \phi^{(2)}_H \in U_H$}.
\end{align*}
Equation \eqref{var1} can be rewritten in as follows,
\begin{align*}
\left\langle\left(\tilde{A}_e + \tilde{A}_i\right)\tilde{u},\phi\right\rangle_{U'\times U} = \left\langle \bar{s}_e - A_i u_m, R_H \phi\right\rangle_{U'_H\times U_H} = \left\langle R_H^{*}\bar{s}_e - R_H^{*}A_i u_m, \phi\right\rangle_{U'\times U},
\end{align*}
thus,
\begin{align*}
\tilde{u} = \left(\tilde{A}_e + \tilde{A}_i\right)^{-1}R_H^{*}\bar{s}_e - \left(\tilde{A}_e + \tilde{A}_i\right)^{-1}R_H^{*}A_i u_m = \left(\tilde{A}_e + \tilde{A}_i\right)^{-1}R^{*}_H\left(\bar{s}_e - A_iu_m\right).
\end{align*}
Substituting this expression for $\tilde{u}$ in \eqref{var0}, we obtain:
\begin{align*}
&\left\langle A_i u_m, \phi_H\right\rangle_{U'_H\times U_H} + \left\langle A_i R_H\left(\tilde{A}_i + \tilde{A}_e\right)^{-1}R^{*}_H\bar{s}_e, \phi_H\right\rangle_{U'_H\times U_H}\\
& - \left\langle A_i R_H\left(\tilde{A}_i + \tilde{A}_e\right)^{-1}R^{*}_HA_i u_m, \phi_H\right\rangle_{U'_H\times U_H} = 0,\;\;\; \hbox{for all $\phi_H \in U_H$}.
\end{align*}
In terms of the weak operators $\tilde{A}_i, \tilde{A}_e, A_i$ the above equation is written as 
\begin{align*}
\bar{A}u_m = \bar{s},
\end{align*}
where $\bar{A}: U_H \rightarrow U'_H$, with 
\begin{align*}
\bar{A} = A_i - A_iR_H\left(\tilde{A}_i + \tilde{A}_e\right)^{-1}R^{*}_HA_i = A_iR_H\left(\tilde{A}_i + \tilde{A}_e\right)^{-1}\tilde{A}_eP_H = A_iR_H\left(\tilde{A}_i + \tilde{A}_e\right)^{-1}R^{*}_HA_e,
\end{align*}
and $\bar{s} \in U'_H$ so that 
\begin{align*}
\bar{s} = -A_iR_H\left(\tilde{A}_i + \tilde{A}_e\right)^{-1}R^{*}_H\bar{s}_e.
\end{align*}
With the objective of defining and studying the full bidomain evolution abstract problem, we must to extend the operator $\bar{A}$ and the functional $\bar{s}$ to $V_H$. Then, we define 
$J: V_H \rightarrow U_H$ given by relation $\phi_H \mapsto \phi_H - \frac{1}{|\Omega_H|}\int_{\Omega_H}\phi_H$, where $|\Omega_H|$ is the volume of $\Omega_H$. Finally, we define 
\begin{align*}
A := J^{*}\bar{A}J : V_H \rightarrow V'_H,
\end{align*}
and 
\begin{align*}
s = J^{*}\bar{s} \in V'_H.
\end{align*}
In order the functions $u \in U$ and $u_m \in U_H$ to be considered weak solution of the problem \eqref{eqnn1}-\eqref{eqnn2} in $V$ and $V_H$ it is necessary to assume that $\left\langle\bar{s}_e,1\right\rangle_{U'_H\times U_H} = 0$, that is,
\begin{align*}
\int_{\Sigma_{endo}}s_e = 0.
\end{align*}
Let us give the following proposition. 
\begin{proposition}\label{Prop1}
The bilinear form associated to the operator $\bar{A}$, denoted by us $\bar{a}$, is defined in $U_H\times U_H$ and has been given by the expression 
\begin{align*}
\bar{a}(u_H,\phi_H) = a_i\left(u_H - R_H\left(\tilde{A}_i + \tilde{A}_e\right)^{-1}R^{*}_HA_iu_H,\phi_H\right),
\end{align*}
and the bilinear form associated to operator $A$ is $a(u_H,\phi_H) = \bar{a}(Ju_H,J\phi_H)$ for all $(u_H,\phi_H) \in V_H\times V_H$. 

Furthermore, $a$ is symmetric, continuous and coercive and there is an increasing sequence $0 = \lambda_0 < \lambda_1 \leq \ldots \leq \lambda_n \leq \ldots$ in $\RR$ and an orthonormal base $\left(\psi_i\right)^{\infty}_{i = 0} \subset L_H$, such that $\psi_i \in V_H$ and 
\begin{align*}
a(\psi_i,v) = \lambda_i(\psi_i,v)\;\;\;\hbox{for all $v \in V_H, i \in \NN$.}
\end{align*}
\end{proposition}
\begin{proof}
Let $u_H, \phi_H \in U_H$ be arbitrary, we have:
\begin{align*}
&\left\langle \bar{A}u_H,\phi_H\right\rangle = \left\langle A_iu_H,\phi_H\right\rangle - \left\langle A_iR_H\left(\tilde{A}_i + \tilde{A}_e\right)^{-1}R^{*}_HA_iu_H,\phi_H\right\rangle = \\
& \left\langle A_i\left(u_H - R_H\left(\tilde{A}_i + \tilde{A}_e\right)^{-1}R^{*}_HA_iu_H\right),\phi_H\right\rangle = a_i\left(u_H - R_H\left(\tilde{A}_i + \tilde{A}_e\right)^{-1}R^{*}_HA_iu_H,\phi_H\right).
\end{align*}
From the above, we prove the first part of Proposition. 

Let us prove that $\bar{a}$ is symmetric. We have that
\begin{align*}
&\bar{a}(u_H,\phi_H) = a_i\left(R_H\left(\tilde{A}_i + \tilde{A}_e\right)^{-1}R^{*}_HA_iu_H,\phi_H\right) = a_i\left(\phi_H,R_H\left(\tilde{A}_i + \tilde{A}_e\right)^{-1}R^{*}_HA_iu_H\right) = \\
& \left\langle R^{*}_H\phi_H,\left(\tilde{A}_i + \tilde{A}_e\right)^{-1}R^{*}_HA_iu_H\right\rangle = \left\langle R^{*}_HA_iu_H,\left(\tilde{A}_i + \tilde{A}_e\right)^{-1}R^{*}_H\phi_H\right\rangle = \left\langle A_iu_H,R_H\left(\tilde{A}_i + \tilde{A}_e\right)^{-1}R^{*}_H\phi_H\right\rangle =\\
& a_i\left(u_H,R_H\left(\tilde{A}_i + \tilde{A}_e\right)^{-1}R^{*}_H\phi_H\right) = a_i\left(R_H\left(\tilde{A}_i + \tilde{A}_e\right)^{-1}R^{*}_H\phi_H,u_H\right) = \bar{a}(\phi_H,u_H).
\end{align*}
In the above equality chain, we have used the fact, which is easy to prove, that
\begin{align*}
\left\langle \phi, \left(\tilde{A}_i + \tilde{A}_e\right)^{-1}u\right\rangle = \left\langle u, \left(\tilde{A}_i + \tilde{A}_e\right)^{-1}\phi\right\rangle,\;\;\; \hbox{for all $u,\phi \in U'$}.
\end{align*}
On the other hand, from a simple calculation we get
\begin{align*}
\bar{a}(u_H,u_H) \geq \frac{m}{3}\left(1 + \frac{m}{2M}\right)\left|u_H\right|^{2}_{U_H},
\end{align*} 
and 
\begin{align*}
\left|\bar{a}(u_H,\phi_H)\right| \leq M\left(1 + \frac{M}{2m}\right)\left|u_H\right|_{U_H}\left|\phi_H\right|_{U_H}.
\end{align*}

The result about the existence of the sequence $\left(\lambda_n\right)^{\infty}_{n = 0}$ and the orthonormal bases $\left(\psi_n\right)^{\infty}_{n = 0} \subset V_H$ can be found in \cite{RAV}.
\end{proof}

Now, we are going to define the strong operators associated to operators $A_{i,e}, \tilde{A}_{i,e}$ and $A$. 

We define 
\begin{align*}
D(A) := \left\{u_H \in H_2(\Omega_H): \nabla u_H\cdot \mathbf{n} = 0\;\; \hbox{in a.e. $\partial \Omega_H$}\right\} \subset L_H,
\end{align*}
We have 
\begin{align*}
D(A_i) = D(A_e) = D(A)/\RR,
\end{align*}
and 
\begin{align*}
u_H \in D(A_{i,e}) \mapsto A_{i,e}u = \nabla \cdot\left(\sigma_{i,e}\nabla u_H\right).
\end{align*}
On the other hand, if we denote by
\begin{align*}
D(\widetilde{A}) = \left\{u \in H_2(\Omega): 
\begin{array}{c}
\nabla R_H u \cdot \mathbf{n} = \nabla R_T u \cdot \mathbf{n} = 0,\\
R_H u = R_T u,
\end{array}
\hbox{in a.e. $\partial \Omega_H$},\;\; \hbox{and}\;\; \nabla R_T u \cdot \mathbf{n} = 0\;\; \hbox{in $\Sigma_T$},
\right\}
\end{align*}
we have that 
\begin{align*}
D(\tilde{A}_{i,e}) = D(\tilde{A}_{i} + \tilde{A}_{e}) = D(\widetilde{A})/\RR.
\end{align*}
where for all $u \in D(\tilde{A}_{i,e})$ we have
\begin{align*}
\tilde{A}_{i,e} u = \nabla \cdot\left(\widetilde{\sigma}_{i,e}\nabla u\right).
\end{align*}
It is possible to prove that $A_{i}$ and $\tilde{A}_{i,e}, \tilde{A}_{i} + \tilde{A}_{e}$ are maximal monotone, self-adjoin and have compact inverse in $L_H/\RR$ and $H/\RR$, respectively.

With respect to the functional $s \in V'_H$, it is known that there is an element $\tilde{s} \in L_H$ such that $\left\|\tilde{s}\right\|_{L_H} = \left\|s\right\|_{V'_H}$ and
\begin{align*}
\left\langle s,u_h \right\rangle = \left(\tilde{s},u\right),\;\;\; \hbox{for all $u \in V_H.$}
\end{align*}
\begin{definition}\label{def1}
We define the strong bidomain operator $A: D(A) \subset H \rightarrow H$ given by the relation $$A u_H = A_iR_H\left(\tilde{A}_i + \tilde{A}_e\right)^{-1}R^{*}_HA_e Ju_H,$$ for all $u_H \in D(A)$.

Furthermore, by $s$ we denote again the element in $L_H$ that it is identified with the functional $s \in V_H$ given in the above comment. 
\end{definition}
Note that, we again have denoted by $A$ this operator, which is well-defined due to two facts.

First, 
\begin{align*}
\left.R^{*}_H\right|_{L_H/\RR}: L_H/\RR \rightarrow H/\RR, 
\end{align*}
and it is given by
\begin{align*}
R^{*}_H u = \left\{
\begin{array}{cc}
u, & \hbox{in $\Omega_H$,}\\
0, & \hbox{in $\Omega_T$.}
\end{array}
\right.
\end{align*}
Furthermore, 
\begin{align*}
R_H\left(\tilde{A}_i + \tilde{A}_e\right)^{-1}: H/\RR \rightarrow D(A)/\RR.
\end{align*}
\begin{theorem}
The strong operator $A$ of Definition \ref{def1} is self-adjoin and maximal-monotone, such that,
\begin{align*}
\left(Au_H, \phi_H\right) = a(u_H,\phi_H),\;\;\; \hbox{for all $u_H,\phi_H \in D(A)$}.
\end{align*}
The sequence $\left(\lambda_n\right)^{\infty}_{n = 0}$ given in Proposition \ref{Prop1} satisfies that each $\lambda_i$ is a eigenvalue of $A$ associated to eigenvector $\psi_i \in D(A)$. Furthermore, one has that 
\begin{align*}
D(A) = \left\{u_H \in L_H: \sum^{\infty}_{n = 0}\lambda^{2}_n\left(\psi_n, u_H\right)^{2} < \infty\right\},\;\; and \;\; Au_H = \sum^{\infty}_{n = 0}\lambda_n\left(\psi_n, u_H\right)\psi_n.
\end{align*}
Finally, we have that 
$A u_m = s, u = \left(\tilde{A}_e + \tilde{A}_i\right)^{-1}R^{*}_H\left(s - A_iu_m\right)$ and $u_e = R_H u \in D(A)$, if and only if $(u,u_m)$ are solutions of problem \eqref{eqnn1}-\eqref{eqnn2} a.e.
\end{theorem}
On the other hand, for any $a_1 > 0$, $A + a_1 : D(A) \rightarrow L_H$ is also a sectorial, auto-adjoin, maximal-monotone operator such that its spectrum consists of the eigenvalues $\left\{a_1 + \lambda_i\right\}^{\infty}_{i = 0}$ with  eigenvectors $\left\{\psi_i\right\}^{\infty}_{i = 0}$. The family of bounded operators $e^{-t(A + a_1)}$ for all $t \geq 0,$ is well defined and $e^{-0(A + a_1)} = I_d$ (the identity operator in $L_H$).

With respect to this family of operators, it is possible to give several properties, such as:
\begin{itemize}
\item for all $t > 0,$ the spectrum of operator $e^{-t(A + a_1)}$ is $\left\{e^{-(\lambda_i + a_1) t}\right\}^{\infty}_{i = 0}$ thus, $1 \not\in \sigma\left(e^{-t(A + a_1)}\right)$. Furthermore, $\left(\psi_i\right)^{\infty}_{i = 0}$ are the eigenvectors associated,
\item $\left(I_d - e^{-T(A + a_1)}\right)^{-1}$ is a linear bounded operator, densely defined. Its spectrum is $\left(\left(1 - e^{-(\lambda_i + a_1)T}\right)^{-1}\right)^{\infty}_{i = 0},$ with eigenvectors $\left(\psi_i\right)^{\infty}_{i = 0},$
\item $\left\|e^{-t(A + a_1)}\right\| \leq 1$, for all $t \geq 0$.
\end{itemize}
\subsection{Definition of the strong solution. Uniqueness and local existence}
In this section, $f,g$ are locally Lipschitz. That is, for all $(u_0,w_0) \in \RR^{2}$, there is an open $U \ni (u_0,w_0)$ and a constant $K$ such that 
\begin{align*}
\begin{array}{c}
|f(u_1,w_1) - f(u_2,w_2)| \leq K|(u_1,w_1) - (u_2,w_2)|,\\
|g(u_1,w_1) - g(u_2,w_2)| \leq K|(u_1,w_1) - (u_2,w_2)|.
\end{array}
\end{align*}
Note, the above means that $f + a_1u$ is also locally Lipschitz continuous, for any $a_1 > 0$.

Consider $\mathcal{Z} := L_H\times L_H$ with the norm $\|(u,w)\|_{\mathcal{Z}} = \max\left\{\|u\|_{L_H},\|w\|_{L_H}\right\}$ which is a Banach space. The operator $\mathcal{A}: D(\mathcal{A}) \subset \mathcal{Z} \rightarrow \mathcal{Z}$ is defined, given by 
\begin{align*}
\mathcal{A}(u,w) = 
\left(
\begin{array}{cc}
a_1 + A & 0\\
0 & 0
\end{array}
\right)
\left(
\begin{array}{c}
u\\
w
\end{array}
\right),
\end{align*}
where $D(\mathcal{A}) = D(A)\times L_H$.
\begin{lemma}
The operator $\mathcal{A}$ is sectorial, has compact resolvent and $\Real \left\{\sigma(\mathcal{A})\right\} > 0$.
\end{lemma}
To deal with the non-linearity, we use the fractional power of the operator $A$, $A^{\alpha}$, and the interpolation spaces $D(A^{\alpha})$ with $\alpha \geq 0$. In this sense, we have $A^{\alpha}: D(A^{\alpha}) \subset L_H \rightarrow L_H$ is an unbounded operator, such that 
\begin{align*}
D(A^{\alpha}) = \left\{u \in L_H: \sum^{\infty}_{i = 0}\lambda^{2\alpha}_i\left(u,\psi_i\right)^{2} < \infty\right\},\;\; \hbox{and}\;\; A^{\alpha}u = \sum^{\infty}_{i = 0}\lambda^{\alpha}_i\left(u,\psi_i\right)\psi_i.
\end{align*} 
The space $D(A^{\alpha})$ with the norm 
\begin{align*}
\left\|u\right\|_{D(A^{\alpha})} = \left\|u + A^{\alpha}u\right\|_{L_H} = \sqrt{\sum^{\infty}_{i = 0}(1 + \lambda^{2\alpha}_i)(u,\psi_i)^{2}},
\end{align*} 
is a Banach space.

We also can define $\mathcal{A}^{\alpha}$, and the interpolation spaces $D(\mathcal{A}^{\alpha})$ as
\begin{align*}
\mathcal{Z}^{\alpha} := D(A^{\alpha})\times D(A^{\alpha}),\;\; \hbox{and}\;\; \mathcal{A}^{\alpha}(u,w) = \left(\sum^{\infty}_{i = 0}\lambda^{\alpha}_i\left(u,\psi_i\right)\psi_i,0\right),\;\;\; \forall (u,w) \in \mathcal{Z}^{\alpha}.
\end{align*} 
The space $\mathcal{Z}^{\alpha}$, equipped with norm $\|(u,w)\|_{\alpha} = \max\left\{\left\|u\right\|_{D(A^{\alpha})},\left\|w\right\|_{D(A^{\alpha})}\right\}$ is also a Banach space. It is possible to prove that for all $0 \leq \alpha < \beta \leq 1,$ $\mathcal{Z}^{\beta} \subset \mathcal{Z}^{\alpha}$ and $\mathcal{Z}^{0} = \mathcal{Z}$, $\mathcal{Z}^{1} = D(\mathcal{A}) \subset H_2(\Omega)\times H_2(\Omega)$. Since the resolvent of $\mathcal{A}$ is compact, the inclusion $\mathcal{Z}^{\beta} \subset \mathcal{Z}^{\alpha}$ is compact, provided that $\alpha < \beta$.

For any $\alpha > 0$, the operator $\mathcal{A}^{-\alpha}$ can also defined in $\mathcal{Z}$. It is linear, bounded, injective, and is the inverse of operator $\mathcal{A}^{\alpha}$. In this case, $\mathcal{Z}^{\alpha}$ coincides with the range of $\mathcal{A}^{-\alpha}$. We define $\mathcal{Z}^{-\alpha}$ to the set $\mathcal{Z}$ with norm
\begin{align*}
\left\|z\right\|_{-\alpha} = \left\|\mathcal{A}^{-\alpha}z\right\|_{\mathcal{Z}},
\end{align*}
which is a Banach space. It is easy to prove that the following inclusions are continuous
\begin{align*}
\mathcal{Z}^{\alpha} \subset \mathcal{Z} \subset \mathcal{Z}^{-\alpha},\;\;\; \hbox{for all $\alpha > 0$}.
\end{align*}
The proof of these results can be found in \cite{Henry}.

If $3/4 < \alpha \leq 1$, then $\mathcal{Z}^{\alpha} \subset L_{\infty}(\Omega_H)\times L_{\infty}(\Omega_H)$ and given that $f + a_1,g: \RR^{2} \rightarrow \RR$ are locally Lipschitz continuous then, $F = (f + a_1,g): \mathcal{Z}^{\alpha} \rightarrow \mathcal{Z}$ is locally Lipschitz continuous. That is, given $z_0 \in \mathcal{Z}^{\alpha}$ there is an open $\mathcal{U} \subset \mathcal{Z}^{\alpha}$ containing to $z_0$ and a positive constants $L$, such that
\begin{align*}
\left\|F(z_1) - F(z_2)\right\|_{\mathcal{Z}} \leq L \left\|z_1 - z_2\right\|_{\alpha},\;\;\; \hbox{for all $z_1, z_2 \in \mathcal{U}$.}
\end{align*} 
Following \cite{Henry}, it is possible to give a definition of strong solution of the Cauchy problem associated to the bidomain problem \eqref{eqn1}-\eqref{eqn9}.
\begin{definition}[Strong solution]\label{def1}
Consider the functions $z(t) = (u(t),w(t)): [0,t_1) \rightarrow \mathcal{Z}$, and $u: [0,t_1) \rightarrow H$ where
\begin{align*}
u(t) = \left\{
\begin{array}{cc}
u_e(t), & \hbox{in $\Omega_H$,}\\
u_T(t), & \hbox{in $\Omega_T$.}
\end{array}
\right.
\end{align*}
We say that the pair $(z,u)$ is a local strong solution to \eqref{eqn1}-\eqref{eqn9} if
\begin{enumerate}
\item $z: [0,t_1) \rightarrow \mathcal{Z}$ is continuous such that $z(0) = z_0 = (u_0,w_0)$, for all $t \in (0,t_1)$, $z(t) \in \mathcal{U} \cap D(\mathcal{A})$, and $\dfrac{dz(t)}{dt}$ exists in the Fr\'echet sense,  
\item for all $t \in (0,t_1)$ the following semilinear evolution equation is satisfied by $z(t)$
\begin{align}\label{semieq1}
\frac{dz(t)}{dt} = -\mathcal{A}z(t) + F(z(t)) + S(t),\;\;\; \hbox{in $\mathcal{Z},$}
\end{align}
and $u(t) = \left(\tilde{A}_e + \tilde{A}_i\right)^{-1}R^{*}_H\left(s - A_i Ju_m(t)\right).$
\end{enumerate}
If $t_1 = \infty$, we say that the solution is global. Furthermore, if the pair $z,u$ is a global solution such that $z(t) = z(t + T), u(t) = u(t + T)$, we say that this is a $T$-periodic solution.
\end{definition}
From Lemma 3.3.2 of \cite{Henry}, it can be demonstrate that if $z(t)$ is a local strong solution with initial condition $z_0$, then   
\begin{align}\label{integraleq1}
z(t) = e^{-t\mathcal{A}}z_0 + \int^{t}_0e^{-(t-\tau)\mathcal{A}}\left(F(z(\tau)) + S(\tau)\right)d\tau,\;\;\; t \in (0,t_1).
\end{align}
Conversely, if $z: [0,t_1) \rightarrow \mathcal{Z}^{\alpha}$ is a continuous function that satisfies the integral equation \eqref{integraleq1}, and 
\begin{align}\label{II2}
\int^{\rho}_0\left\|F(z(\tau)) + S(\tau)\right\|_{\mathcal{Z}}d\tau < \infty,\;\;\; \hbox{for some $\rho > 0$,}
\end{align} 
then $z(t)$ is a strong solution in $(0,t_1)$. 

The following theorem is a consequence of Theorem 3.2.2 of \cite{Henry}.
\begin{theorem}
If the initial condition $(u_0,w_0) \in \mathcal{Z}^{\alpha}$, then the bidomain model has a only strong solution, $(u,w)$, that satisfies the initial condition $(u(0),w(0)) = (u_0,w_0).$
\end{theorem} 
\section{Faedo-Galerkin system associated to the operational formulation}
In this section, we investigate when the Faedo-Galerkin approximations converge to a strong solution, in the sense of Definition \ref{def1}. We obtain the Faedo-Galerkin systems from the semilinear evolution equation \eqref{semieq1}. Our idea is based in the fact that Faedo-Galerkin approximations have the same regularity of a strong solution. Furthermore, we can achieve that these converge in a suitable sense such that its limit be a strong solution. We believe that it is an important result because, until now, only weak solutions were obtained as the limit of Faedo-Galerkin approximations.
 
We define $H_m = \left\langle \psi_0,\psi_1,\ldots,\psi_m\right\rangle$ as the subspace generated by the $m+1$-first eigenvectors associated to the bidomain operator $A$, endowed with the norm of $L_H$. We denote by $\mathcal{Z}_m = H_m\times H_m$. 

Note, $\mathcal{Z}_m \subset \mathcal{Z}^{\alpha}$, for all $\alpha \geq 0$ which is only considered as a sets inclusion. We can define $\mathcal{Z}^{\alpha}_m = H_m\times H_m$ endowed with the norm in $\mathcal{Z}^{\alpha}$, for $\alpha \geq 0$.

We can define $\mathcal{P}_m: \mathcal{Z} \rightarrow \mathcal{Z}_m \subset \mathcal{Z}$ by 
\begin{align*}
\mathcal{P}_m z = \left(\sum^{m}_{i = 0}\left(u,\psi_i\right)\psi_i,\sum^{m}_{i = 0}\left(w,\psi_i\right)\psi_i\right),\;\;\; \forall z = (u,w) \in \mathcal{Z},\;\;\; \hbox{for all $m \in \NN\cup \left\{0\right\}.$}
\end{align*}
The operator that was defined in the  previous equation is called $\mathcal{Z}$ in $\mathcal{Z}_m$ projection operator, which satisfies that $\left\|\mathcal{P}_m\right\|_{\mathcal{L}(\mathcal{Z})} \leq 1$. Similarly, it can be easily proved that $$\left.\mathcal{P}_m\right|_{\mathcal{Z}^{\alpha}}: \mathcal{Z}^{\alpha} \rightarrow \mathcal{Z}^{\alpha},$$ $\left\|\mathcal{P}_m\right\|_{\mathcal{L}(\mathcal{Z}^{\alpha})} \leq 1$, and
\begin{align*}
\left\|\mathcal{P}_m z - z\right\|_{\mathcal{Z}^{\alpha}} \rightarrow 0,\;\;\left\|\mathcal{P}_m z - z\right\|_{\mathcal{Z}} \rightarrow 0\;\;\; \hbox{when $m \rightarrow \infty$, for all $z \in \mathcal{Z}^{\alpha}.$}
\end{align*} 
In other words, $\left\|\mathcal{P}_m - I_d\right\|_{\mathcal{L}(\mathcal{Z})} \rightarrow 0$ and  $\left\|\mathcal{P}_m - I_d\right\|_{\mathcal{L}(\mathcal{Z}^{\alpha})} \rightarrow 0$. Note that, this implies
\begin{align*}
\left\|\mathcal{P}_m - \mathcal{P}_n\right\|_{\mathcal{L}(\mathcal{Z})} \leq \left\|I_d - \mathcal{P}_n\right\|_{\mathcal{L}(\mathcal{Z})} + \left\|I_d - \mathcal{P}_m\right\|_{\mathcal{L}(\mathcal{Z})} \rightarrow 0,
\end{align*}
and 
\begin{align*}
\left\|\mathcal{P}_m - \mathcal{P}_n\right\|_{\mathcal{L}(\mathcal{Z})^{\alpha}} \rightarrow 0
\end{align*}
when $m,n \rightarrow \infty$. 

Define the functions 
\begin{align*}
u_m(t,x) = \sum^{m}_{i = 0}u^{(m)}_i(t)\psi_i(x),\;\; w_m(t,x) = \sum^{m}_{i = 0}w^{(m)}_i(t)\psi_i(x),
\end{align*}
where the vector functions 
\begin{align}\label{U_m}
U_m(t) = \left(u^{(m)}_0(t),u^{(m)}_1(t),\ldots,u^{(m)}_m(t)\right)\;\; \hbox{and}\;\; W_m(t) = \left(w^{(m)}_0(t),w^{(m)}_1(t),\ldots,w^{(m)}_m(t)\right)
\end{align}
are in $C^{1}\left(I,\RR^{m+1}\right)$ for certain interval $I \subseteq \RR$ containing to $0$. Note, $z_m(t) = (u_m(t),w_m(t)) \in \mathcal{Z}_m$, in particular $z_m(t) \in D(\mathcal{A})$  for all $t \in I$. Besides, 
\begin{align*}
u'_m(t,x) = \sum^{m}_{i = 0}u^{(m)'}_i(t)\psi_i(x),\;\; w'_m(t,x) = \sum^{m}_{i = 0}w^{(m)'}_i(t)\psi_i(x),
\end{align*}
are the Fr\'echet derivatives of the function $u_m, w_m$, respectively, thus, $z'_m(t) = (u'_m(t),w'_m(t)) \in \mathcal{Z}_m$.

In that follows, we are interested in to study the following problem of initial conditions: to find a function $z_m \in C^{1}(I;\mathcal{Z}^{\alpha}_m)$, for certain $0 < \alpha < 1$, which satisfies  
\begin{align}\label{eqforz_m}
\left\{
\begin{array}{c}
\dfrac{dz_m(t)}{dt} = -\mathcal{A}z_m(t) + \mathcal{P}_mF(z_m(t)) + \mathcal{P}_mS(t),\\
z_m(0) = z^{(0)}_m \in \mathcal{Z}_m^{\alpha}.
\end{array}
\right.
\end{align}
and converges in a suitable form, which will be specified later, to a strong solution in the sense of Definition \ref{def1}. Take into account that $z^{(0)}_m = \left(u^{(0)}_m,w^{(0)}_m\right)$, such that,
\begin{align*}
u^{(0)}_m = \sum^{m}_{i = 0}u^{(m,0)}_i\psi_i,\;\; w^{(0)}_m = \sum^{m}_{i = 0}w^{(m,0)}_i\psi_i,
\end{align*}
thus, to the initial condition $z^{(0)}_m$ we can also associate a pair of vectors of $\RR^{m+1}$
\begin{align}\label{U_0}
U^{(0)}_m = \left(u^{(m,0)}_0,u^{(m,0)}_1,\ldots,u^{(m,0)}_m\right),\;\;W^{(0)}_m = \left(w^{(m,0)}_0,w^{(m,0)}_1,\ldots,w^{(m,0)}_m\right).
\end{align}
The sequence $\left\{z_m\right\}^{\infty}_{m = 0} \subset C^{1}(I;\mathcal{Z}^{\alpha}_m)$ built from the solution of problem \eqref{eqforz_m} for each $m \in \NN$, will be called by us as strong Faedo-Galerkin approximations of the strong solution of the torso-coupled bidomain model. Now, we will show the reason why this terminology is used. 
\begin{remark}
In order to achieve this convergence, we suppose that the sequence of initial conditions $\left\{z^{(0)}_m\right\}$ converges in $\mathcal{Z}^{\alpha}$ and its limit is $z_0 \in \mathcal{Z}^{\alpha}$. In this case, there are constants $m_0 \in \NN, \delta_{m_0} > 0$ and an open $\mathcal{U}_{m_0} \subset \mathcal{Z}^{\alpha}$, such that $F$ is Lipzchitz continuous in $\mathcal{U}_{m_0}$ with Lipschitz constant $L_{m_0}$ and
\begin{align*}
\left\{z^{(0)}_m\right\}^{\infty}_{m=m_0} \subset \mathcal{B}_{\alpha}\left(z_0,\frac{\delta_{m_0}}{2}\right) \subset \mathcal{U}_{m_0}.
\end{align*}
Now, for each of the remaining elements of the sequence, $\left\{z^{(0)}_k\right\}^{m_0-1}_{k = 0}$, there is an open $\mathcal{U}_{k} \subset \mathcal{Z}^{\alpha}$ and a ball $\mathcal{B}_{\alpha}\left(z^{(0)}_k,\frac{\delta_{k}}{2}\right)$, such that,
\begin{align*}
\mathcal{B}_{\alpha}\left(z^{(0)}_k,\frac{\delta_{k}}{2}\right) \subset \mathcal{U}_{k},
\end{align*}
and $F$ is Lipschitz continuous in $\mathcal{U}_k$ with Lipschitz constant $L_k$. Then, we have
\begin{align*}
\left\{z^{(0)}_m\right\}^{\infty}_{m=0} \subset \left(\bigcup^{m_0-1}_{m = 0}\mathcal{B}_{\alpha}\left(z^{(0)}_k,\frac{\delta}{2}\right)\right) \bigcup  \mathcal{B}_{\alpha}\left(z_0,\frac{\delta}{2}\right) \subset \mathcal{U} = \bigcup^{m_0}_{m = 0}\mathcal{U}_{m}
\end{align*}
where $\delta = \min\left\{\delta_{0},\delta_{1},\ldots,\delta_{m_0}\right\}$, and $F$ is Lipschitz continuous, with Lipschitz constant $L = \max\left\{L_0,L_1,\ldots,L_{m_0}\right\}$.
\end{remark}
Summarizing the presented above, we assume that $\left\{z^{(0)}_m\right\}$ satisfies the following condition
\begin{description}
\label{CB}\item[CB] There is a finite family of balls $\left\{\mathcal{B}_{\alpha}\left(z_j,\frac{\delta}{2}\right)\right\}^{m_0}_{j=0}$, where $z_j \in \mathcal{Z}^{\alpha}$ for $j = 0,1,\ldots,m_0$, such that, $$\left\{z^{(0)}_m\right\} \subset \mathcal{F} = \bigcup^{m_0}_{j=0}\mathcal{B}_{\alpha}\left(z_j,\frac{\delta}{2}\right),$$ and $F$ is Lipschitz continuous in $\mathcal{F}$.
\end{description}
If $z_m$ satisfies the problem in \eqref{eqforz_m}, then the vectors $U_m(t), W_m(t)$ associated with $z_m(t) =  (u_m(t),w_m(t))$ given in \eqref{U_m}, satisfy the following problem of initial conditions associated to a system of $2m+2$ equations 
\begin{align}\label{FGsystem}
\left\{
\begin{array}{c}
u^{(m)'}_i(t) = -\left(a_1 + \lambda_i\right)u^{(m)}_i(t) + \displaystyle\int_{\Omega_H}f(u_m(t),w_m(t))\psi_i + \int_{\Omega_H}S(\tau)\psi_i,\\
w^{(m)'}_i(t) = \displaystyle\int_{\Omega_H}g(u_m(t),w_m(t))\psi_i,\\
U_m(0) = U_m^{(0)},\;\; W_m(0) = W_m^{(0)}.
\end{array}
\right. i = 0,\ldots,m.
\end{align}
This means that the Cauchy problem \eqref{eqforz_m} is equivalent to \eqref{FGsystem}, which is associated to a system of ordinary \\ differential equation. That is, if \eqref{FGsystem} has an only solution, then, \eqref{FGsystem} does, too. Besides, if the solution of \eqref{FGsystem} is defined for all $t \geq 0$, then, the same happens with the solution of \eqref{eqforz_m}.

On the other hand, $z_m$ is solution of \eqref{eqforz_m} if and only if 
\begin{align}\label{eqintegral1}
z_m(t) = e^{-t\mathcal{A}}z^{(0)}_m + \int^{t}_0e^{-(t-\tau)\mathcal{A}}\mathcal{P}_m\left[F(z_m(\tau)) + S(\tau)\right]d\tau.
\end{align}
In what follows, we set $3/4 < \alpha_0 \leq 1$, and we are going to prove that the Faedo-Galerkin problem \eqref{eqforz_m} has a only solution $z_m$ in $L_{2}(0,t_1;\mathcal{Z}^{\alpha_0})$, for certain $t_1 > 0$, when the sequence of initial conditions satisfies the condition \textbf{CB}.

By $C\left([0,t_1];\mathcal{Z}_m^{\alpha_0}\right)$ we denote the space of functions $z_m(t)$ continuous in $[0,t_1]$ with values in $\mathcal{Z}_m^{\alpha_0}$, and 
\begin{align*}
V_{\delta} = \left\{z_m \in C\left([0,t_1];\mathcal{Z}^{\alpha_0}\right): z_m(t) \in \mathcal{F},\;\; \hbox{for $t \in [0,t_1]$}\right\}.
\end{align*}
By $\mathcal{K}_m: V_{\delta} \subset C\left([0,t_1];\mathcal{Z}^{\alpha_0}\right)\rightarrow C\left([0,t_1];\mathcal{Z}^{\alpha_0}\right)$, we denote the operator given by the relation
\begin{align*}
K_m (z_m)(t) = e^{-t\mathcal{A}}z^{(0)}_m + \int^{t}_0e^{-(t-\tau)\mathcal{A}}\mathcal{P}_m\left[F(z_m(\tau)) + S(\tau)\right]d\tau,\;\;\; z_m \in V_{\delta}.
\end{align*}
Now, $\delta$ and $t_1$ are chosen, such that $K_m\left(V_{\delta}\right) \subseteq V_{\delta}$. In fact, if $z_m \in V_{\delta}$, then, for any $t \in [0,t_1]$, $\left\|z_m(t) - \widetilde{z}_t\right\|_{\alpha_0} \leq \frac{\delta}{2}$, where $\widetilde{z}_t \in \left\{z_0, z_1,\ldots,z_{m_0}\right\}$.

We have
\begin{align}\label{temp1}
\nonumber
&\left\|\left(K_m z_m\right)(t) - \widetilde{z}_t\right\|_{\alpha_0} \leq \left\|\left(K_m z_m\right)(t) - z^{(0)}_m\right\|_{\alpha_0} + \left\|z^{(0)}_m - \widetilde{z}_t\right\|_{\alpha_0}\\
&\leq \left\|\left(e^{-t\mathcal{A}} - I_d\right)z^{(0)}_m\right\|_{\alpha_0} + \left\|\int^{t}_0e^{-(t-\tau)\mathcal{A}}\mathcal{P}_m\left[F(z_m(\tau)) + S(\tau)\right]d\tau\right\|_{\alpha_0} + \frac{\delta}{2}.
\end{align}
Let us estimate the first term of the right-side of \eqref{temp1}. In this sense, we write 
\begin{align*}
(I_d + \mathcal{A}^{\alpha_0})\left(e^{-t\mathcal{A}} - I_d\right)z^{(0)}_m = -\int^{t}_{0}(I_d + \mathcal{A}^{\alpha_0})\left(e^{-s\mathcal{A}} - I_d\right)z^{(0)}_m ds =  -\int^{t}_{0}\left(e^{-s\mathcal{A}} - I_d\right)(I_d + \mathcal{A}^{\alpha_0})z^{(0)}_m ds,
\end{align*}
thus, 
\begin{align*}
\left\|\left(e^{-t\mathcal{A}} - I_d\right)z^{(0)}_m\right\|_{\alpha_0} \leq \int^{t_1}_0\left\|\left(e^{-s\mathcal{A}} - I_d\right)\right\|_{\mathcal{L}(\mathcal{Z})}\left\|z^{(0)}_m\right\|_{\alpha_0}ds \leq 2\left\|z^{(0)}_m\right\|_{\alpha_0}t_1.
\end{align*}
We choose $t_1 > 0$, such that 
\begin{align*}
\sup_{t \in [0,t_1]}\left\|\left(e^{-t\mathcal{A}} - I_d\right)z^{(0)}_m\right\|_{\alpha_0} \leq \frac{\delta}{4}.
\end{align*} 

To estimate the second term of the right-side of \eqref{temp1}, let us see
\begin{align*}
&\left\|\int^{t}_0e^{-(t-\tau)\mathcal{A}}\mathcal{P}_m\left[F(z_m(\tau)) + S(\tau)\right]d\tau\right\|_{\alpha_0} \leq 
\left\|\int^{t_1}_0\left(I_d + A^{\alpha_0}\right)e^{-(t-\tau)\mathcal{A}}\mathcal{P}_m\left[F(z_m(\tau)) + S(\tau)\right]d\tau\right\|_{\mathcal{Z}} \\
&\leq \int^{t_1}_0\left\|\left(I_d + A^{\alpha_0}\right)e^{-(t-\tau)\mathcal{A}}\mathcal{P}_m\left[F(z_0) + S(\tau)\right]\right\|_{\mathcal{Z}}d\tau + \int^{t_1}_0\left\|\left(I_d + A^{\alpha_0}\right)e^{-(t-\tau)\mathcal{A}}\mathcal{P}_m\left(F(z_m(\tau)) - F(z_0)\right)\right\|_{\mathcal{Z}}d\tau\\
&\leq M(B + L\delta)\int^{t_1}_0\tau^{-\alpha_0}e^{a_1 \tau}d\tau \leq (B + L\delta)\frac{e^{bt_1}t^{1 - \alpha_0}_1}{1-\alpha_0},
\end{align*}
where we have taken into account that 
\begin{align*}
\left\|F(z_0) + S(\tau)\right\|_{\mathcal{Z}} \leq B,
\end{align*}
for certain constant $B > 0$. In summary, 
\begin{align*}
\sup_{t \in [0,t_1]}\left\|\int^{t}_0e^{-(t-\tau)\mathcal{A}}\mathcal{P}_mF(z_m(\tau))d\tau\right\|_{\alpha_0} \leq M(B + L\delta)\frac{e^{a_1t_1}t^{1 - \alpha_0}_1}{1-\alpha_0} \leq \frac{\delta}{4},
\end{align*}
for certain $\delta, t_1$ appropriately chosen. We have proved that, there are positive constants $\delta, t_1$, such that $\mathcal{K}_m(V_{\delta}) \subseteq V_{\delta}$.

In a similar way, it can also be proved that $\mathcal{K}_m$ is a contracting map if we suitably choose $t_1, \delta$. The previously proved results can be statemented in the following lemma. 
\begin{lemma}\label{lemaexistence1}
The Cauchy problem \eqref{eqforz_m} has an only solution $z_m \in L_{2}(0,t_1;\mathcal{Z}^{\alpha_0})$ for certain $t_1 > 0$ and $3/4 < \alpha_0 < 1$. Besides, the sequence $\left\{z_m\right\}^{\infty}_{m = 0} \subset \mathcal{B}_{L_{2}(0,t_1;\mathcal{Z}^{\alpha_0})}(r_0)$ for certain $r_0 > 0$, where
\begin{align*}
\mathcal{B}_{L_{2}(0,t_1;\mathcal{Z}^{\alpha_0})}(r_0) = \left\{z \in L_{2}(0,t_1;\mathcal{Z}^{\alpha_0}): \left\|z\right\|_{L_{2}(0,t_1;\mathcal{Z}^{\alpha_0})} \leq r_0\right\}.
\end{align*}
\end{lemma}
\begin{proof}
With the previous explanation, the existence of a solution of problem \eqref{eqforz_m} $z_m \in V_{\delta}  \subset C([0,t_1];\mathcal{Z}_m^{\alpha_0})$, has been demonstrated for each $m \in \NN\cup\left\{0\right\}$. Note that $C([0,t_1];\mathcal{Z}_m^{\alpha_0}) \subset L_{2}(0,t_1;\mathcal{Z}^{\alpha_0})$, continuously. Moreover, $z_m \in V_\delta$, implies that 
\begin{align*}
\left\|z_m(t)\right\|_{\alpha_0} \leq \max\{\left\|z_0\right\|_{\alpha_0},\left\|z_1\right\|_{\alpha_0},\ldots,\left\|z_{m_0}\right\|_{\alpha_0}\} + \frac{\delta}{2},
\end{align*}
thus
\begin{align*}
\left(\int^{t_1}_{0}\left\|z_m(\tau)\right\|^{2}_{\alpha_0}d\tau\right)^{\frac{1}{2}} \leq \sqrt{t_1}\left(\max\{\left\|z_0\right\|_{\alpha_0},\left\|z_1\right\|_{\alpha_0},\ldots,\left\|z_{m_0}\right\|_{\alpha_0}\} + \frac{\delta}{2}\right).
\end{align*}
That is, $r_0 = \sqrt{t_1}\left(\max\{\left\|z_0\right\|_{\alpha_0},\left\|z_1\right\|_{\alpha_0},\ldots,\left\|z_{m_0}\right\|_{\alpha_0}\} + \frac{\delta}{2}\right)$.
\end{proof}

Now, we shall demonstrate the following Proposition.
\begin{proposition}\label{PropCauchy}
Let $\left\{z_m\right\}^{\infty}_{m = 0} \subset V_{\delta}$ be the sequence of solution of the Faedo-Galerkin systems obtained in Lemma \ref{lemaexistence1}, is a Cauchy sequence in $C\left(0,t_1;\mathcal{Z}^{\alpha_0}\right)$. 
\end{proposition}
\begin{proof}
We have
\begin{align*}
&\left\|z_m(t) - z_n(t)\right\|_{\alpha_0} \leq \left\|e^{-t\mathcal{A}}\left(z^{(0)}_m - z^{(0)}_n\right)\right\|_{\alpha_0} + \int^{t}_0\left\|e^{-(t - \tau)\mathcal{A}}\left(\mathcal{P}_mF(z_m(\tau)) - \mathcal{P}_nF(z_n(\tau))\right)\right\|_{\alpha_0}d\tau\\
&+ \int^{t}_0\left\|e^{-(t - \tau)\mathcal{A}}\left(\mathcal{P}_m - \mathcal{P}_n\right)S(\tau)\right\|_{\alpha_0}d\tau \\
&\leq \left\|z^{(0)}_m - z^{(0)}_n\right\|_{\alpha_0} + \left\|\mathcal{P}_m - \mathcal{P}_n\right\|_{\mathcal{L}(\mathcal{Z}^{\alpha_0})}\widehat{S}\frac{t^{1-\alpha_0}_1}{1 - \alpha_0} +  \int^{t}_0\left\|e^{-(t - \tau)\mathcal{A}}\left(\mathcal{P}_mF(z_m(\tau)) - \mathcal{P}_nF(z_n(\tau))\right)\right\|_{\alpha_0}d\tau.
\end{align*}
Now, 
\begin{align*}
&\int^{t}_0\left\|e^{-(t - \tau)\mathcal{A}}\left(\mathcal{P}_mF(z_m(\tau)) - \mathcal{P}_nF(z_n(\tau))\right)\right\|_{\alpha_0}d\tau \leq \int^{t}_0e^{-(t - \tau)a_1}(t - \tau)^{-\alpha_0}\left\|\mathcal{P}_mF(z_m(\tau)) - \mathcal{P}_nF(z_n(\tau))\right\|_{\mathcal{Z}}d\tau\\
&\leq \int^{t}_0(t - \tau)^{-\alpha_0}\left\|\mathcal{P}_mF(z_m(\tau)) - \mathcal{P}_nF(z_n(\tau))\right\|_{\mathcal{Z}}d\tau \leq \int^{t}_0(t - \tau)^{-\alpha_0}\left\|\mathcal{P}_m\left(F(z_m(\tau)) - F(z_n(\tau))\right)\right\|_{\mathcal{Z}}d\tau \\
&+ \int^{t}_0(t - \tau)^{-\alpha_0}\left\|\left(\mathcal{P}_m - \mathcal{P}_n\right)F(z_n(\tau))\right\|_{\mathcal{Z}}d\tau \leq \left(L\max_{t \in [0,t_1]}\left\|z_m(t) - z_n(t)\right\|_{\alpha_0} + \frac{\delta}{2}\left\|\mathcal{P}_m - \mathcal{P}_n\right\|_{\mathcal{L}(\mathcal{Z})} + \cdots \right. \\ 
&\left.\left\|\mathcal{P}_m - \mathcal{P}_n\right\|_{\mathcal{L}(\mathcal{Z})}\max\left\{\left\|F(z_0)\right\|_{\mathcal{Z}},\ldots,\left\|F(z_{m_0})\right\|_{\mathcal{Z}}\right\}\right)\frac{t^{1 - \alpha_0}_1}{1 - \alpha_0}.
\end{align*}
Taking into account the previous estimations, we obtain the following
\begin{align}\label{convelo1}
\nonumber
\left\|z_m - z_n\right\|_{C([0,t_1;\mathcal{Z}^{\alpha_0})} \leq \frac{1}{1 - L\frac{t^{1-\alpha_0}_1}{1-\alpha_0}}\left(\left\|z^{(0)}_m - z^{(0)}_n\right\|_{\alpha_0} + \left\|\mathcal{P}_m - \mathcal{P}_n\right\|_{\mathcal{L}(\mathcal{Z})}\widehat{S}\frac{t^{1-\alpha_0}_1}{1 - \alpha_0} \right.\\
\left. + \left\|\mathcal{P}_m - \mathcal{P}_n\right\|_{\mathcal{L}(\mathcal{Z})}\max\left\{\left\|F(z_0)\right\|_{\mathcal{Z}},\ldots,\left\|F(z_{m_0})\right\|_{\mathcal{Z}}\right\}\frac{t^{1 - \alpha_0}_1}{1 - \alpha_0}\right),
\end{align}
where $t_1$ is taken such that $\frac{t^{1-\alpha_0}_1}{1-\alpha_0} < L$.
\end{proof}

Now, we can give the following Theorem.
\begin{theorem}\label{Theo3}
If the sequence of the initial conditions $\left\{z^{(0)}_m\right\}^{\infty}_{m = 0}$ converges to $z_0$ in $\mathcal{Z}^{\alpha_0}$, with $3/4 < \alpha_0 < 1$, then the sequence of solutions of the problem \eqref{eqforz_m}, $z_m \in V_{\delta}$ converges in $C(0,t_1;\mathcal{Z}^{\alpha_0})$ to function $z \in C(0,t_1;\mathcal{Z}^{\alpha_0})$, which is the only strong solution of the coupled-torso bidomain model, such that $z(0) = z_0$.
\end{theorem}
\begin{proof}
By Proposition \ref{PropCauchy}, we can affirm that sequence of strong Faedo-Galerkin approximations is a sequence of Cauchy in $C(0,t_1,\mathcal{Z}^{\alpha_0})$. Furthermore, we can choose $t_1$ and $\delta > 0$, such that $\left\{z_m\right\}^{\infty}_{m = 0} \subset V_{\delta}$, which is a closed subspace of $C(0,t_1,\mathcal{Z}^{\alpha_0})$, thus there is $z \in V_{\delta}$, such that
\begin{align*}
z_m \rightarrow z,\;\;\; \hbox{strongly in $C(0,t_1,\mathcal{Z}^{\alpha_0})$.}
\end{align*}
To demonstrate that $z$ is a strong solution, we must only prove that it satisfies \eqref{integraleq1} and \eqref{II2}. In first place, we have
\begin{align*}
z_m(t) = e^{-t\mathcal{A}}z^{(0)}_m + \int^{t}_0e^{-(t-\tau)\mathcal{A}}\mathcal{P}_m\left[F(z_m(\tau)) + S(\tau)\right]d\tau,\;\;\; \hbox{for each $m = 0,1,\cdots,$ and $t \in [0,t_1].$}
\end{align*}
It is very easy to prove that by making the limit on each sum of the equality it is obtained 
\begin{align*}
z(t) = e^{-t\mathcal{A}}z_0 + \int^{t}_0e^{-(t-\tau)\mathcal{A}}\mathcal{P}_m\left[F(z(\tau)) + S(\tau)\right]d\tau,\;\;\; t \in [0,t_1].
\end{align*}
On the other hand, since $z \in V_{\delta}$, it is possible to prove that 
\begin{align*}
\int^{t_1}_0\left\|F(z(\tau)) + S(\tau)\right\|_{\mathcal{Z}}d\tau < \infty.
\end{align*}
\end{proof}

Theorem \ref{Theo3} tells us that the only strong solution of a Cauchy problem associated to the torso-coupled bidomain model with initial condition $z_0 \in \mathcal{Z}^{\alpha_0}$, it is obtained as the limit of a sequence of strong Faedo-Galerkin approximations; which are obtained as solutions of the Cauchy problem \eqref{eqforz_m}, where the initial condition assumed is $z^{(0)}_m = \mathcal{P}_mz_0$. Note that \eqref{eqforz_m} is essentially a system of ordinary differential equations, which has important implications from both a theoretical and computational point of view. In the first place, if each $z_m$ is defined for all $t \geq 0$ so is the strong solution and if each $z_m$ is a $T$-periodic function, then its limit also is $T$-periodic. Computationally, this result provides a method of numerical resolution of the torso-coupled bidomain model which is theoretically supporting. Upon achieving convergence, we can take $z_m$ as an approximation of the solution sought by making $m$ large enough. The inequality \eqref{convelo1} provides an estimate of the velocity of convergence of the sequence of strong Faedo-Galerkin approximations.
\section{Existence of strong periodic solution}
In this section, we demonstrate the existence a strong $T$-periodic solution, in the sense of Definition \ref{def1}, assuming that $S$ is $T$-periodic, which is continuous from $[0,T]$ in $\mathcal{Z}$. 

From now on, we will use the following notation
\begin{align*}
&C_T = \left\{z \in C([0,T];\mathcal{Z}): z(0) = z(T)\right\},\\
&C^{\alpha_0}_T = \left\{z \in C([0,T];\mathcal{Z}^{\alpha_0}): z(0) = z(T)\right\},
\end{align*}
where, in each case, we consider the following norms
\begin{align*}
&\left\|f\right\|_{C_T} = \sup_{t \in [0,T]}\left\|f(t)\right\|_{\mathcal{Z}},\\
&\left\|f\right\|_{C^{\alpha_0}_T} = \sup_{t \in [0,T]}\left\|f(t)\right\|_{\mathcal{Z}^{\alpha_0}},
\end{align*}
for $f \in C_T$ or $f \in C^{\alpha_0}_T$, respectively.

As the previous section, this $T$-periodic solution is obtained as a limit of a sequence of solutions of system \eqref{eqforz_m} which are $T$-periodic.

Let 
\begin{align}\label{lineareq1}
\frac{dz_m}{dt} = -\mathcal{A}z_m + \mathcal{P}_mS
\end{align}
be the linear equation obtained from \eqref{eqforz_m}, and let $z_m \in C([0,\infty);\mathcal{Z}^{\alpha_0}_m)$ be a solution \eqref{lineareq1}. Then,
\begin{align*}
z_m(t) = e^{-t\mathcal{A}}z^{(0)}_m + \int^{t}_0e^{-(t - \tau)\mathcal{A}}\mathcal{P}_mS(\tau)d\tau,
\end{align*}
where $z_m(0) = z^{(0)}_m$. Due to the uniqueness of the solution of the Cauchy problem associated with \eqref{lineareq1} and the periodicity of the function $S$, it is possible to prove that $z_m$ is $T$-periodic if and only if $z_m(0) = z_m(T)$. Thus
\begin{align*}
z_m(T) = e^{-T\mathcal{A}}z^{(0)}_m + \int^{T}_0e^{-(T - \tau)\mathcal{A}}\mathcal{P}_mS(\tau)d\tau = z^{(0)}_m,
\end{align*}
that is,
\begin{align*}
\left(I_d - e^{-T\mathcal{A}}\right)z_m^{(0)} = \int^{T}_0e^{-(T - \tau)\mathcal{A}}\mathcal{P}_mS(\tau)d\tau.
\end{align*}
Let us remember that $0 \notin \sigma(\mathcal{A}) \Rightarrow 1 \notin \sigma\left(e^{-T\mathcal{A}}\right)$, thus the operator $I_d - e^{-T\mathcal{A}}$ is invertible and 
\begin{align*}
\left\|\left(I_d - e^{-T\mathcal{A}}\right)^{-1}\right\|_{\mathcal{L}(\mathcal{Z})} \leq \frac{1}{1 - e^{-Ta_1}}.
\end{align*}
Thus, we obtain
\begin{align*}
z_m^{(0)} = \left(I_d - e^{-T\mathcal{A}}\right)^{-1}\int^{T}_0e^{-(T - \tau)\mathcal{A}}\mathcal{P}_mS(\tau)d\tau.
\end{align*}
Finally, we can affirm that if $z_m$ is $T$-periodic, then
\begin{align*}
z_m(t) = \int^{T}_0\mathcal{R}(t,\tau)\mathcal{P}_mS(\tau)d\tau,
\end{align*}
where
\begin{align*}
\mathcal{R}(t,\tau) = \left\{
\begin{array}{cc}
\left(I_d - e^{-T\mathcal{A}}\right)^{-1}e^{-(t-\tau)\mathcal{A}}, & \hbox{if $0 \leq \tau \leq t \leq T,$}\\
\left(I_d - e^{-T\mathcal{A}}\right)^{-1}e^{-(t + T -\tau)\mathcal{A}}, & \hbox{if $0 \leq t \leq \tau \leq T.$}
\end{array}
\right.
\end{align*}
For convenience, for $f \in C_T$ we denote by
\begin{align*}
\mathcal{L}^{(p)}(f)(t) = \int^{T}_0\mathcal{R}(t,\tau)f(\tau)d\tau.
\end{align*}
Note that, in this case, $z_m$ is a $T$-periodic solution of \eqref{lineareq1} if and only if
\begin{align*}
z_m(t) = \mathcal{L}^{(p)}(\mathcal{P}_m S)(t),\;\;\; t \in [0,T].
\end{align*}
Before continuing with the results that interest us, let us give the following Lemma.
\begin{lemma}\label{Lemmaimpor2}
Let $f$ be a function in $C_T$, then $\mathcal{L}^{(p)}_m(f) \in C^{\alpha_0}_T$, and
\begin{align*}
\left\|\mathcal{L}^{(p)}(f)\right\|_{C^{\alpha_0}_T} \leq 2\left(\frac{1}{a_1} + \frac{T^{1-\alpha_0}}{(1-\alpha_0)\left(1 - e^{-a_1T}\right)}\right)\left\|f\right\|_{C_T}.
\end{align*} 
\end{lemma}
\begin{proof}
If $f \in C^{\alpha_0}_T$
\begin{align*}
\left\|\mathcal{L}^{(p)}(f)(t)\right\|_{\alpha_0} \leq \int^{t}_0\left\|\left(I_d - e^{-T\mathcal{A}}\right)^{-1}e^{-(t-\tau)\mathcal{A}}f(\tau)\right\|_{\alpha_0}d\tau + \int^{T}_t\left\|\left(I_d - e^{-T\mathcal{A}}\right)^{-1}e^{-(t + T -\tau)\mathcal{A}}f(\tau)\right\|_{\alpha_0}d\tau.
\end{align*}
Now, 
\begin{align*}
&\int^{t}_0\left\|\left(I_d - e^{-T\mathcal{A}}\right)^{-1}e^{-(t-\tau)\mathcal{A}}f(\tau)\right\|_{\alpha_0}d\tau = \int^{t}_0\left\|\left(I_d - e^{-T\mathcal{A}}\right)^{-1}\left(I_d + A^{\alpha_0}\right)e^{-(t-\tau)\mathcal{A}}f(\tau)\right\|_{\mathcal{Z}}d\tau \\
&\leq \int^{t}_0\left\|\left(I_d - e^{-T\mathcal{A}}\right)^{-1}e^{-(t-\tau)\mathcal{A}}f(\tau)\right\|_{\mathcal{Z}}d\tau + \int^{t}_0\left\|\left(I_d - e^{-T\mathcal{A}}\right)^{-1}A^{\alpha_0}e^{-(t-\tau)\mathcal{A}}f(\tau)\right\|_{\mathcal{Z}}d\tau\\
&\leq \frac{\left\|f\right\|_{C^{\alpha_0}_T}}{1 - e^{-a_1T}}\int^{t}_0e^{-a_1(t - \tau)}d\tau +  \frac{\left\|f\right\|_{C^{\alpha_0}_T}}{1 - e^{-a_1T}}\int^{t}_0(t - \tau)^{-\alpha_0}d\tau\\
&\leq \frac{\left\|f\right\|_{C^{\alpha_0}_T}}{1 - e^{-a_1T}}\int^{T}_0e^{-a_1(t - \tau)}d\tau +  \frac{\left\|f\right\|_{C^{\alpha_0}_T}}{1 - e^{-a_1T}}\int^{T}_0(t - \tau)^{-\alpha_0}d\tau \\
& \leq \left(\frac{1}{a_1} + \frac{T^{1-\alpha_0}}{(1-\alpha_0)\left(1 - e^{-a_1T}\right)}\right)\left\|f\right\|_{C^{\alpha_0}_T}.
\end{align*}
Similarly, it is possible to demonstrate
\begin{align*}
\int^{T}_t\left\|\left(I_d - e^{-T\mathcal{A}}\right)^{-1}e^{-(t + T - \tau)\mathcal{A}}f(\tau)\right\|_{\alpha_0}d\tau  \leq \left(\frac{1}{a_1} + \frac{T^{1-\alpha_0}}{(1-\alpha_0)\left(1 - e^{-a_1T}\right)}\right)\left\|f\right\|_{C^{\alpha_0}_T}.
\end{align*}
\end{proof}

Now, consider the equation in \eqref{eqforz_m}
\begin{align}\label{impeq3}
\frac{dz_m}{dt} = -\mathcal{A}z_m + \mathcal{P}_mF(z_m) + \mathcal{P}_mS.
\end{align}
Following the same line of thought, it shows that $z_m$ is a $T$-periodic solution of \eqref{impeq3} if it satisfies the integral equation 
\begin{align*}
z_m(t) = \mathcal{L}^{(p)}\mathcal{P}_m\left(F(z_m) + S\right)(t),\;\;\; t \in [0,T].
\end{align*}
This means that we will search these periodic solutions as the fixed points of the operator 
\begin{align*}
\mathcal{K}^{(p)}_m(z_m)(t) := \mathcal{L}^{(p)}\mathcal{P}_m\left(F(z_m) + S\right)(t),
\end{align*}
defined from $C^{\alpha_0}_T$ in itself.

We consider a ball $\mathcal{B}_{\alpha_0}(r_0) \subset \mathcal{Z}^{\alpha_0}$, such that $F: \mathcal{B}_{\alpha_0}(r_0) \rightarrow \mathcal{Z}$ be Lipschitz continuous with Lipschitz constant $L$. Next, we shall prove it is possible suitably to choose $T, r_0$, such that the operator $\mathcal{K}^{(p)}_m$ has an only fixed point in 
\begin{align*}
C^{(\alpha_0,r_0)}_T = \left\{z_m \in C([0,T];\mathcal{B}_{\alpha_0}(r_0)): z_m(0) = z_m(T)\right\},
\end{align*}
which is endowed with the norm in $C^{\alpha_0}_T$.

We suppose that $z_m \in C^{(\alpha_0,r_0)}_T$, taking into account Lemma \ref{Lemmaimpor2} we have
\begin{align*}
\left\|\mathcal{K}^{(p)}_m(z_m)\right\|_{C^{\alpha_0}_T} \leq  2\left(\frac{1}{a_1} + \frac{T^{1-\alpha_0}}{(1-\alpha_0)\left(1 - e^{-a_1T}\right)}\right)\left\|F(z_m) + S\right\|_{C_T}.
\end{align*}
On the other hand, we have
\begin{align*}
\left\|F(z_m(t)) + S(t)\right\|_{\mathcal{Z}} \leq \left\|F(z_m(t)) - F(0)\right\|_{\mathcal{Z}} + \left\|F(0)\right\|_{\mathcal{Z}} + \widehat{S} \leq L\left\|z_m(t)\right\|_{\alpha_0} + \left\|F(0)\right\|_{\mathcal{Z}} + \widehat{S},
\end{align*}
where $\widehat{S} = \|S\|_{C_T}$, from which it is deduced that 
\begin{align*}
\left\|F(z_m) + S\right\|_{C_T} \leq Lr_0 + \left\|F(0)\right\|_{\mathcal{Z}} + \widehat{S}
\end{align*}
From the above, we obtain
\begin{align}\label{desigualdad1}
\left\|\mathcal{K}^{(p)}_m(z_m)\right\|_{C^{\alpha_0}_T} \leq 2\left(Lr_0 + \|F(0)\|_{\mathcal{Z}} + \widehat{S}\right)\left(\frac{1}{a_1} + \frac{T^{1-\alpha_0}}{(1 - \alpha_0)(1 - e^{-a_1T})}\right).
\end{align}
From a similar analysis, we have
\begin{align}\label{desigualdad2}
\left\|\mathcal{K}^{(p)}_m\left(z^{(1)}_m\right) - \mathcal{K}^{(p)}_m\left(z^{(2)}_m\right)\right\|_{C^{\alpha_0}_T} \leq 2\left(\frac{1}{a_1} + \frac{T^{1-\alpha_0}}{(1 - \alpha_0)\left(1 - e^{-a_1T}\right)}\right)Lr_0\left\|z^{(1)}_m - z^{(2)}_m\right\|_{C^{\alpha_0}_T}.
\end{align}
By the inequalities \eqref{desigualdad1} and \eqref{desigualdad2}, we can affirm that $\mathcal{K}^{(p)}_m(C^{(\alpha_0,r_0)}_T) \subseteq C^{(\alpha_0,r_0)}_T$ if 
\begin{align}\label{desigualdad3}
\left(Lr_0 + \|F(0)\|_{\mathcal{Z}} + \widehat{S}\right)\left(\frac{1}{a_1} + \frac{T^{1-\alpha_0}}{(1 - \alpha_0)(1 - e^{-a_1T})}\right) \leq \frac{r_0}{2},
\end{align}
and the map $\mathcal{K}^{(p)}_m$ is a contraction when
\begin{align}\label{desigualdad4}
\left(\frac{1}{a_1} + \frac{T^{1-\alpha_0}}{(1 - \alpha_0)\left(1 - e^{-a_1T}\right)}\right)Lr_0 < \frac{1}{2}.
\end{align}
\begin{theorem}
If conditions \eqref{desigualdad3}, \eqref{desigualdad4} are assumed, and besides
\begin{align*}
\frac{1}{a_1} + \frac{T^{1-\alpha_0}}{(1-\alpha_0)\left(1-e^{-a_1T}\right)} < \frac{1}{2},
\end{align*}
then, there is a sequence of strong Faedo-Galerkin approximations $\left\{z^{(p)}_m\right\}^{\infty}_{m = 0} \subset C^{(\alpha_0,r_0)}_T$, whose elements are $T$-periodic functions. Furthermore, the sequence converge to a $T$-periodic function $z^{(p)} \in C^{(\alpha_0,r_0)}_T$, which is a strong solution of the coupled-torso bidomian model in the sense of Definition \ref{def1}.
\end{theorem}
\begin{proof}
The inequalities \eqref{desigualdad3} and \eqref{desigualdad4}, imply the existence of only fixed point of $\mathcal{K}^{(p)}_m$ in $C^{(\alpha_0,r_0)}_T$, denoted by $z^{(p)}_m$, which is a $T$-periodic solution of \eqref{impeq3}. 

Now, we prove that $\left\{z^{(p)}_m\right\}^{\infty}_{m = 0}$ is a Cauchy sequence in $C^{\alpha_0}_T$. In fact, we have
\begin{align*}
&\left\|z_m - z_n\right\|_{C^{\alpha_0}_T} = \left\|\mathcal{K}^{(p)}_m(z_m) - \mathcal{K}^{(p)}_n(z_n)\right\|_{C^{\alpha_0}_T} = \left\|\mathcal{L}^{(p)}\left(\mathcal{P}_mF(z_m) - \mathcal{P}_nF(z_n)\right)\right\|_{C^{\alpha_0}_T} \\
&\leq \left\|\mathcal{L}^{(p)}\mathcal{P}_m\left(F(z_m) - F(z_n)\right)\right\|_{C^{\alpha_0}_T} + \left\|\mathcal{L}^{(p)}\left(\mathcal{P}_m - \mathcal{P}_n\right)(F(z_n) - F(0))\right\|_{C^{\alpha_0}_T} + \left\|\mathcal{L}^{(p)}\left(\mathcal{P}_m - \mathcal{P}_n\right)F(0)\right\|_{C^{\alpha_0}_T}\\
&\leq 2\left(\frac{1}{a_1} + \frac{T^{1-\alpha_0}}{(1-\alpha_0)\left(1 - e^{-a_1T}\right)}\right)\left(L\left\|z_m - z_n\right\|_{C^{\alpha_0}_T} + \left\|\mathcal{P}_m - \mathcal{P}_n\right\|_{\mathcal{L}(\mathcal{Z})}(Lr_0 + \|F(0)\|_{\mathcal{Z}})\right),
\end{align*}
thus
\begin{align*}
\left\|z_m - z_n\right\|_{C^{\alpha_0}_T} \leq \frac{Lr_0 + \|F(0)\|_{\mathcal{Z}}}{1 - 2\left(\frac{1}{a_1} + \frac{T^{1-\alpha_0}}{(1-\alpha_0)\left(1 - e^{-a_1T}\right)}\right)} \left\|\mathcal{P}_m - \mathcal{P}_n\right\|_{\mathcal{L}(\mathcal{Z})}.
\end{align*}
The right-side of the above inequality converges towards to zero, thus, it is demonstrated that the sequence $\left\{z^{(p)}_m\right\}^{\infty}_{m = 0}$ is of Cauchy in $C^{\alpha_0}_T$. Furthermore, since $C^{(\alpha_0,r_0)}_T$ is closed in $C^{\alpha_0}_T$, being this last a Banach space, we obtain that 
\begin{align*}
z^{(p)}_m \rightarrow z^{(p)} \in C^{(\alpha_0,r_0)}_T.
\end{align*}
It is easy to see that $z^{(p)}$ is a $T$-periodic strong solution of the torso-coupled bidomain model. 
\end{proof}
%\section*{Acknowledgments}

\bibliography{mybibfile}

\end{document}